\documentclass[11pt]{amsart}
\usepackage{amsmath, amssymb, amsfonts, verbatim, graphicx, amsthm, tikz-cd, mathrsfs, color, comment, fancyhdr, amsrefs}
\usepackage[top=1in, bottom=1in, left=1in, right=1in]{geometry}
\usepackage{enumerate}

\usepackage{mathtools}

\usepackage{url}

\usepackage{enumitem}

\usepackage[symbol]{footmisc}

\usepackage[foot]{amsaddr}

\allowdisplaybreaks[1]

\usepackage[sc]{mathpazo}
\usepackage[T1]{fontenc}

\newcommand{\vertiii}[1]{{\left\vert\kern-0.25ex\left\vert\kern-0.25ex\left\vert #1
		\right\vert\kern-0.25ex\right\vert\kern-0.25ex\right\vert}}

\newcommand{\A}{\mathbb{A}}
\newcommand{\B}{\mathbb{B}}

\def\Xint#1{\mathchoice
	{\XXint\displaystyle\textstyle{#1}}%
	{\XXint\textstyle\scriptstyle{#1}}%
	{\XXint\scriptstyle\scriptscriptstyle{#1}}%
	{\XXint\scriptscriptstyle\scriptscriptstyle{#1}}%
	\!\int}
\def\XXint#1#2#3{{\setbox0=\hbox{$#1{#2#3}{\int}$ }
		\vcenter{\hbox{$#2#3$ }}\kern-.6\wd0}}

\def\dashint{\Xint-}

\usepackage{setspace}

\usepackage{hyperref}

\usepackage[nameinlink]{cleveref}

\crefname{Thm}{Theorem}{Theorems}
\crefname{Prop}{Proposition}{Propositions}
\crefname{Lem}{Lemma}{Lemmas}

\theoremstyle{plain}
\newtheorem{Thm}{Theorem}
\newtheorem{Prop}{Proposition}
\newtheorem{Lem}[Prop]{Lemma}

\newtheorem{Qn}[Prop]{Question}

\newtheorem*{namedthm}{\namedthmname}
\newcounter{namedthm}



\theoremstyle{definition}
\newtheorem{Def}[Prop]{Definition}
\newtheorem{Not}[Prop]{Notation}
\newtheorem{Rmk}[Prop]{Remark}

\newtheorem{Conv}[Prop]{Convention}

\newcommand{\C}{\mathbb{C}}

\newcommand{\Z}{\mathbb{Z}}
\newcommand{\N}{\mathbb{N}}

\newcommand{\R}{\mathbb{R}}

\renewcommand{\epsilon}{\varepsilon}

\author{Aidan Young}

\title{An ergodic Lebesgue differentiation theorem}

\email{\url{aidanyoungmath@gmail.com}}
	
\begin{document}
	
\begin{abstract}
	We show that if $(X, \mu, T)$ is a probability measure-preserving dynamical system, and $\mathscr{P}$ is a countable partition of $(X, \mu)$, then the limit
	\[
	\lim_{n, k \to \infty} \mathbb{E} \left[ \frac{1}{k} \sum_{j = 0}^{k - 1} f \circ T^j \mid \bigvee_{i = 0}^{n - 1} T^{-i} \mathscr{P} \right]
	\]
	exists almost surely for all $f \in L^p(\mu), p > 1$. We prove this as a corollary of a geometric result: that if $(X, \mu)$ is a metric measure space on which the Hardy-Littlewood maximal inequality holds, then the limit
	\[\lim_{r \searrow 0, k \to \infty} \frac{1}{\mu(B(x, r))} \int_{B(x, r)} \frac{1}{k} \sum_{j = 0}^{k - 1} f \circ T^j \mathrm{d} \mu\]
	exists almost surely.
	
	In response to the first version of this article, J. Li has improved upon these results in a recent work, relaxing the assumption that $f \in L^p(\mu)$ for some $p > 1$ to the assumption that $f \in L \log L(\mu)$, and relaxing some of the geometric assumptions we place on the underlying probability space. We will summarize his improvements.
	\end{abstract}
	
	\maketitle
	
	\section*{Introduction}
	
	Consider a probability measure-preserving dynamical system $(X, \mu, T)$, where $X$ is a metric space and $\mu$ is a Borel probability measure with respect to that metric topology. Sequences of the form
	\begin{equation}\label{eq:TSD}
		\left( \mu(F_k)^{-1} \int_{F_k} \frac{1}{k} \sum_{j = 0}^{k - 1} f \circ T^j \mathrm{d} \mu \right)_{k = 1}^\infty,
	\end{equation}
	where $(F_k)_{k = 1}^\infty$ is a sequence of measurable subsets of $X$ with positive measure, were introduced in \cite{Assani-Young} under the name of ``spatial-temporal differentiations." These get their name from the way that they combine a ``spatial average" (i.e. the average over a subset $F_k$ of the space $(X, \mu)$) with a ``temporal average" (i.e. the ergodic average, which can be understood as an average over time). In \cite{Assani-Young}, we inquired about these averages' limiting behaviors in various special cases. We also considered a related question: if we treat the sequence $(F_k)_{k = 1}^\infty$ in \eqref{eq:TSD} as somehow being ``randomly chosen," then can we say something about the limiting behavior of this sequence \eqref{eq:TSD} in a probabilistic sense? In particular, we studied limits of this form where the $F_k$ are chosen to be a sequence of neighborhoods of a point $x \in X$ with rapidly decaying diameter, and found that under suitable continuity conditions on the map $T$, the sequence \eqref{eq:TSD} would converge a.s. for continuous $f$.
	
	Here, we expand this investigation to consider measure-preserving (but not necessarily continuous) transformations $T$, as well as $f \in L^p(\mu)$ for $p > 1$, which need not be continuous. We conclude the following:
	
	\begin{Thm}\label{thm:martingale theorem}
	Let $(X, \mu, T)$ be a probability measure-preserving system, and $\mathscr{P}$ a countable measurable partition of $(X, \mu)$. Consider $f \in L^p(\mu)$ for $p > 1$. Then
	\begin{align*}
	\lim_{n, k \to \infty} \mathbb{E} \left[ \frac{1}{k} \sum_{j = 0}^{k - 1} f \circ T^j \mid \bigvee_{i = 0}^{n - 1} T^{-i}\mathscr{P} \right]	& = \mathbb{E} \left[ f \mid \mathfrak{I} \cap \mathfrak{P} \right]	& \textrm{a.s.,}
	\end{align*}
	where $\mathfrak{I}$ is the $\sigma$-algebra of $T$-invariant measurable subsets of $X$, and $\mathfrak{P}$ is the $\sigma$-algebra generated by $\bigcup_{n = 1}^\infty \bigvee_{i = 0}^{n - 1} T^{-i} \mathscr{P}$.
	\end{Thm}
	
	Although \Cref{thm:martingale theorem} has \emph{prima facie} no geometry to it, we prove it as a special case of \Cref{thm:Ergodic Lebesgue Differentiation}, which is a kind of `'ergodic Lebesgue differentiation theorem."
	
	\begin{Thm}\label{thm:Ergodic Lebesgue Differentiation}
	Consider a probability measure-preserving dynamical system $(X, \mu, T)$, where $X$ is a metric space, and $\mu$ is a Borel probability measure on $X$ such that the associated centered Hardy-Littlewood maximal operator is of weak type $(1, 1)$, i.e. there exists a constant $C \geq 1$ such that
	\begin{align*}
	\mu \left( \left\{ x \in X : \sup_{r > 0} \mu(B(x, r)) \int_{B(x, r)} |f| \mathrm{d} \mu \geq \epsilon \right\} \right)	& \leq \frac{C}{\epsilon} \|f\|_1	& \textrm{for all $f \in L^1(\mu), \epsilon > 0$.}
	\end{align*}
	Then for all $f \in L^p(\mu), p > 1$, we have
	\begin{align*}
	\lim_{r \searrow 0, k \to \infty} \mu(B(x, r))^{-1} \int_{B(x, r)} \frac{1}{k} \sum_{j = 0}^{k - 1} f \circ T^j \mathrm{d} \mu	& = \mathbb{E} \left[ f \mid \mathfrak{I} \right]	& \textrm{a.s.},
	\end{align*}
	where $\mathfrak{I}$ is the $\sigma$-algebra of $T$-invariant measurable subsets of $X$.
	\end{Thm}
	
	For the sake of brevity, we call this condition on $(X, \mu)$ the \emph{Hardy-Littlewood property} (cf. \Cref{def:HLP}). This property is known to hold for several classes of metric measure spaces. Classically, it holds for all doubling metric measure spaces $(X, \mu)$ (cf. \cite[Theorem 2.2]{HeinonenLectures}), which encompasses, for example, the torus $\R^n / \Z^n$ with the standard Euclidean metric and Haar probability measure. This property is also known to hold for all Borel probability measures on $\R$, but not all Borel probability measures on $\R^2$ (cf. \cite{SjogrenMaximalFunction}). While the Hardy-Littlewood property implies that the (centered) Lebesgue differentiation theorem holds on that space, the converse fails. For example, while not all Borel probability measures on $\R^2$ have the Hardy-Littlewood property, they all accommodate the Lebesgue differentiation theorem, i.e. $\mu(B(x, r))^{-1} \int_{B(x, r)} f \mathrm{d} \mu \stackrel{r \searrow 0}{\to} f(x)$ a.s. for all $f \in L^1(\mu)$ (cf. \cite[Theorem 5.8.8]{Bogachev1}). However, a classical result of S. Sawyer (cf. \cite[Theorem 1]{SawyerMaximalInequality}) provides conditions under which the Lebesgue differentiation theorem can imply the Hardy-Littlewood property, namely when the group of $\mu$-preserving isometries of $X$ acts ergodically on $(X, \mu)$.
	
	In response to the first version of this article, J. Li has shown in \cite{SharpOrliczEndpoints} that \Cref{thm:Ergodic Lebesgue Differentiation} can be strengthened in two different ways. First, we can replace the Hardy-Littlewood property hypothesis in \Cref{thm:Ergodic Lebesgue Differentiation} with the weaker hypothesis that a centered Lebesgue differentiation theorem holds on the space, and second, we can replace the assumption that $f \in L^p(\mu)$ for some $p > 1$ with the assumption that $f \in L \log L(\mu)$. We discuss Li's work in more detail in \Cref{sec:Open questions}.
	
	\Cref{sec:Proof of main theorem} will be devoted to proving \Cref{thm:Ergodic Lebesgue Differentiation}. In \Cref{sec:martingale case}, we will prove \Cref{thm:martingale theorem} as a special case of \Cref{thm:Ergodic Lebesgue Differentiation}. In \Cref{sec:Open questions}, we pose some questions left open by \Cref{thm:Ergodic Lebesgue Differentiation}, including a note of J. Li's subsequent solutions to these questions in \cite{SharpOrliczEndpoints}. Finally, in \Cref{sec:exotic averages}, we prove a version of \Cref{thm:Ergodic Lebesgue Differentiation} for subsequential ergodic averages along the square, and comment on how the techniques we employ there can be generalized to other pointwise good subsequential or weighted ergodic averages.
	
	\section{Proving \Cref{thm:Ergodic Lebesgue Differentiation}}\label{sec:Proof of main theorem}
	
	The naive ``proof" of \Cref{thm:Ergodic Lebesgue Differentiation} is that it should be a synthesis of the Lebesgue differentiation theorem and the pointwise ergodic theorem, i.e.
	\begin{align*}
		\mu(B(x, r))^{-1} \int_{B(x, r)} \frac{1}{k} \sum_{j = 0}^{k - 1} f \circ T^j \mathrm{d} \mu	& \approx \frac{1}{k} \sum_{j = 0}^{k - 1} f \left( T^j x \right)	& \textrm{[Lebesgue differentiation theorem]} \\
		& \approx f^* (x)	& \textrm{[pointwise ergodic theorem]}	& ,
	\end{align*}
	where we write $f^* : = \mathbb{E} \left[ f \mid \mathfrak{I} \right]$. A variation on this argument is used at several points in \cite{Assani-Young}, where instead of using the Lebesgue differentiation theorem, we perform the first approximation by using more explicit estimates based on the continuity properties of $T$ and $f$, on the condition that the radii $r = r_k$ depend on $k$ and satisfy a ``rapid decay" condition determined by $T$. In this article, however, we use more classical measure-theoretic techniques which let us get rid of this rapid decay hypothesis.
	
	Our proof of \Cref{thm:Ergodic Lebesgue Differentiation} has a classical structure. We consider the operators $U_{r, k} f(x) : = \mu(B(x, r))^{-1} \int_{B(x, r)} \frac{1}{k} \sum_{j = 0}^{k - 1} f \circ T^j \mathrm{d} \mu$, and
	\begin{enumerate}
		\item establish a dense set $\mathcal{G} \subseteq L^p(\mu)$ for which $U_{r, k} g \stackrel{r \searrow 0, k \to \infty}{\to} g^*$ a.s. for all $g \in \mathcal{G}$ (cf. \Cref{lem:L infty convergence}),
		\item find a maximal inequality for the operator family $(U_{r, k})_{r > 0, k \in \N}$ (cf. \Cref{lem:Facts about HLP}), and finally
		\item use a standard Banach principle argument to complete the proof.
	\end{enumerate}
	
	Before proceeding, we establish some notation and terminology that we will use throughout the article.
	
	\begin{Conv}
	Throughout this article, we consider a probability measure-preserving dynamical system $(X, \mu, T)$. Given a measurable set $E \subseteq X$ with positive measure $\mu(E) > 0$, an integrable function $f \in L^1(\mu)$, as well as $k \in \N$, we write
	\begin{align*}
	\dashint_E f \mathrm{d} \mu	& : = \mu(E)^{-1} \int_E f \mathrm{d} \mu ,	& \A_k f	& : = \frac{1}{k} \sum_{j = 0}^{k - 1} f \circ T^j .
	\end{align*}
	We also define three related maximal operators:
	\begin{align*}
		\mathbf{H}^* f(x)	& : = \sup_{r > 0} \dashint_{B(x, r)} |f| \mathrm{d} \mu ,	& \mathbf{A}^* f(x)	& : = \sup_{k \in \N} \A_k |f|(x) ,	& \mathbf{U}^* f(x)	& = \sup_{r > 0} \sup_{k \in \N} U_{r, k} |f| (x) .
	\end{align*}
	\end{Conv}
	
	\begin{Def}\label{def:HLP}
		Consider a Borel probability measure $\mu$ on a metric space $X$. We say that $\mu$ has the \emph{Hardy-Littlewood property} if the \emph{Hardy-Littlewood maximal operator} $\mathbf{H}^*$ is of weak type $(1, 1)$, i.e. there exists a constant $C \geq 1$ such that
	\begin{align*}
		\mu \left( \left\{ x \in X : \mathbf{H}^* f(x) \geq \epsilon \right\} \right)	& \leq \frac{C}{\epsilon} \|f\|_1	& \textrm{for all $f \in L^1(\mu), \epsilon > 0$.}
	\end{align*}
	\end{Def}
	
	\begin{Rmk}
	Elsewhere in the literature, there are many related operators that are given the name ``Hardy-Littlewood maximal operators." For example, in the context of a general metric measure space, we could consider an uncentered Hardy-Littlewood maximal operator $\mathbf{H}_u^* f(x) = \sup_{r > 0} \sup_{y \in B(x, r)} \dashint_{B(y, r)} |f| \mathrm{d} \mu$. More generally, one could define a maximal operator associated to an arbitrary differentiation basis, as in \cite[II-2]{Guzman}. In order to minimize the need for specialized definitions, we will only consider the centered Hardy-Littlewood maximal operator with respect to balls in this article, although the techniques we use to prove \Cref{thm:Ergodic Lebesgue Differentiation} can be easily generalized along these lines.
	\end{Rmk}
	
	\begin{Lem}\label{lem:Facts about HLP}
	Consider a Borel probability measure $\mu$ on a metric space $X$. If $\mu$ has the Hardy-Littlewood property, then
		\begin{enumerate}
			\item $\dashint_{B(x, r)} f \mathrm{d} \mu \stackrel{r \searrow 0}{\to} f(x)$ a.s. for all $f \in L^1(\mu)$;
			\item for all $p > 1$, there exists a constant $C_p \geq 1$ such that $\left\| \mathbf{H}^* f \right\|_p \leq C_p \|f\|_p$ for all $f \in L^p(\mu)$; and
			\item for all $p > 1$, there exists a constant $D_p \geq 1$ such that $\left\| \mathbf{U}^* f \right\|_p \leq D_p \|f\|_p$ for all $f \in L^p(\mu)$, and this constant is independent of $T$.
		\end{enumerate}
	\end{Lem}
	
	\begin{proof}
	A proof of these first two claims can be found in most standard measure theory texts, e.g. \cite[Section 2]{HeinonenLectures}. The last claim can be proven by observing that $\mathbf{U}^* \leq \mathbf{H}^* \mathbf{A}^*$, then combining the second claim of this lemma with the dominated ergodic theorem (cf. \cite[Theorem 6.3 of Chapter 1]{Krengel}), which says that $\left\| \mathbf{A}^* f \right\|_p \leq \frac{p}{p - 1} \|f\|_p$ for all $p > 1, f \in L^p(\mu)$.
	\end{proof}
	
	\begin{Conv}
	For the remainder of \Cref{sec:Proof of main theorem}, we assume that $X$ is a metric space, and $\mu$ is a Borel probability measure on $X$ with the Hardy-Littlewood property.
	\end{Conv}
	
	\begin{Lem}\label{lem:Kernel of invariant expectation}
	Consider a function $f \in L^\infty(\mu)$. Then
		\begin{align*}
			U_{r, k} \left( f - f^* \right)	& \stackrel{r \searrow 0, k \to \infty}{\to} 0	& \textrm{a.s.}
		\end{align*}
	\end{Lem}
	
	\begin{proof}
		We can assume without loss of generality that $f^* = 0$ and $\|f\|_\infty \leq 1$. We can then estimate that for fixed $K \in \N$, we have that $\left\| \left( \A_{k} \A_K f \right) - \left( \A_k f \right) \right\|_\infty \stackrel{k \to \infty}{\to} 0 $. Therefore, for any $K \in \N$, we have
		\begin{align*}
		\left| U_{r, k} f(x) \right|	& = \left| \dashint_{B(x, r)} \A_k f \mathrm{d} \mu \right| \\
			& \leq \left| \dashint_{B(x, r)} \A_k \A_K f \mathrm{d} \mu \right| + \left\| \A_k \A_K f - \A_k f \right\|_\infty \\
			& \leq \dashint_{B(x, r)} \mathbf{A}^* \left( \A_K f \right) \mathrm{d} \mu + \left\| \left( \A_{k} \A_K f \right) - \left( \A_k f \right) \right\|_\infty \\
		\Rightarrow \limsup_{r \searrow 0, k \to \infty} \left| U_{r, k} f(x) \right| & \leq \limsup_{r \searrow 0, k \to \infty} \dashint_{B(x, r)} \mathbf{A}^* \A_K f \mathrm{d} \mu \\
			& = \limsup_{r \searrow 0} \dashint_{B(x, r)} \mathbf{A}^* \A_K f \mathrm{d} \mu \\
			& \leq \mathbf{H}^* \mathbf{A}^* \A_K f (x)	& \textrm{for all $K \in \N$.}
		\end{align*}
		Thus for all $\epsilon > 0$, we have
		\begin{align*}
		\mu \left( \left\{ x \in X : \limsup_{r \searrow 0, k \to \infty} \left| U_{r, k} f(x) \right| \geq \epsilon \right\} \right)	& \leq \left( \frac{\left\| \mathbf{H}^* \mathbf{A}^* \A_K f \right\|_2}{\epsilon} \right)^2 \\
			& \leq \frac{4 C_2^2}{\epsilon^2} \left\| \A_K f \right\|_2^2	& \textrm{for all $K \in \N$} \\
		\stackrel{\textrm{mean ergodic theorem}}{\Rightarrow} \mu \left( \left\{ x \in X : \limsup_{r \searrow 0, k \to \infty} \left| U_{r, k} f(x) \right| \geq \epsilon \right\} \right)	& = 0 ,
		\end{align*}
		where $C_2$ is the constant of the same name from \Cref{lem:Facts about HLP}. Therefore $U_{r, k} f(x) \stackrel{r \searrow 0, k \to \infty}{\to} 0$ a.s.
		\end{proof}
	
	\begin{Lem}\label{lem:L infty convergence}
	Consider a function $f \in L^\infty(\mu)$. Then
	\begin{align*}
		U_{r, k} f	& \stackrel{r \searrow 0, k \to \infty}{\to} f^*	& \textrm{a.s.}
	\end{align*}
	\end{Lem}
	
	\begin{proof}
	If $f \in L^\infty(\mu)$, then
	\begin{align*}
	U_{r, k} f = U_{r, k} \left( f - f^* \right) + U_{r, k} \left( f^* \right) = U_{r, k} \left( f - f^* \right) + \dashint_{B(\cdot, r)} f^* \mathrm{d} \mu \stackrel{r \searrow 0, k \to \infty}{\to} f^*	&	& \textrm{a.s.,}
	\end{align*}
	where $U_{r, k} \left( f - f^* \right) \stackrel{r \searrow 0, k \to \infty}{\to} 0$ by \Cref{lem:Kernel of invariant expectation} and $\dashint_{B(\cdot, r)} f^* \mathrm{d} \mu \stackrel{r \searrow 0}{\to} f^*$ by \Cref{lem:Facts about HLP}.
	\end{proof}
	
	\begin{proof}[Proof of \Cref{thm:Ergodic Lebesgue Differentiation}]
	Fix $f \in L^p(\mu), p > 1$, and consider an arbitrary $g \in L^\infty(\mu)$. Then
	\begin{align*}
	\limsup_{r \searrow 0, k \to \infty} \left| U_{r, k} f(x) - f^*(x) \right|	& = \limsup_{r \searrow 0, k \to \infty} \left| U_{r, k} (f - g)(x) - (f - g)^*(x) + U_{r, k} g(x) - g^*(x) \right| \\
		\textrm{[\Cref{lem:L infty convergence}]}	& = \limsup_{r \searrow 0, k \to \infty} \left| U_{r, k} (f - g)(x) - (f - g)^*(x) \right| \\
			& \leq \mathbf{U}^*(f - g)(x) + \mathbf{A}^* (f - g) (x) \\
	\Rightarrow \left\| \limsup_{r \searrow 0, k \to \infty} \left| U_{r, k} f(x) - f^*(x) \right| \right\|_p	& \leq \left\| \left( \mathbf{U}^* + \mathbf{A}^* \right) (f - g) \right\|_p \\
		\textrm{[\Cref{lem:Facts about HLP}]}	& \leq \left( D_p + \frac{p}{p - 1} \right) \|f - g\|_p ,
	\end{align*}
	where $D_p$ is the constant of the same name from \Cref{lem:Facts about HLP}, and these limits are taken a.s. For all $\epsilon > 0$, we have
	\begin{align*}
	\mu \left( \left\{ x \in X : \limsup_{r \searrow 0, k \to \infty} \left| U_{r, k} f(x) - f^*(x) \right| \geq \epsilon \right\} \right)	& \leq \frac{\left( D_p + \frac{p}{p - 1} \right)^p}{\epsilon^p} \|f - g\|_p^p ,
	\end{align*}
	where this estimate is taken using Chebyshev's inequality. But we can choose $g \in L^\infty(\mu)$ so as to make $\|f - g\|_p$ arbitrarily small, so
	\begin{align*}
	\mu \left( \left\{ x \in X : \limsup_{r \searrow 0, k \to \infty} \left| U_{r, k} f(x) - f^*(x) \right| \geq \epsilon \right\} \right)	& = 0	& \textrm{for all $\epsilon > 0$,}
	\end{align*}
	implying that $U_{r, k} f \stackrel{r \searrow 0, k \to \infty}{\to} f^*$ a.s.
	\end{proof}
		
	\section{Proving \Cref{thm:martingale theorem}}\label{sec:martingale case}
		
	We will prove \Cref{thm:stronger martingale theorem}, which has \Cref{thm:martingale theorem} as a special case. Note that in this section, we do \emph{not} assume that $X$ is endowed with a metric.
	
	\begin{Not}
	Given two partitions $\mathscr{P}, \mathscr{Q}$ of a set $X$, we write $\mathscr{P} \preceq \mathscr{Q}$ to denote that $\mathscr{Q}$ is finer than $\mathscr{P}$, meaning that every element of $\mathscr{Q}$ is contained in an element of $\mathscr{P}$, i.e. $\forall Q \in \mathscr{Q} \; \exists P \in \mathscr{P} \; (Q \subseteq P)$.
	\end{Not}
	
	\begin{Thm}\label{thm:stronger martingale theorem}
		Let $(X, \mu, T)$ be a probability measure-preserving system, and let $\left( \mathscr{P}^{(n)} \right)_{n = 1}^\infty$ be a sequence of countable measurable partitions of $(X, \mu)$ such that $\mathscr{P}^{(1)} \preceq \mathscr{P}^{(2)} \preceq \cdots$. Consider $f \in L^p(\mu)$ for $p > 1$. Then
	\begin{align*}
		\lim_{n, k \to \infty} \mathbb{E} \left[ \frac{1}{k} \sum_{j = 0}^{k - 1} f \circ T^j \mid \mathscr{P}^{(n)} \right]	& = \mathbb{E} \left[ f \mid \mathfrak{I} \cap \mathfrak{P} \right]	& \textrm{a.s.,}
	\end{align*}
		where $\mathfrak{I}$ is the $\sigma$-algebra of $T$-invariant measurable subsets of $X$, and $\mathfrak{P}$ is the $\sigma$-algebra generated by $\bigcup_{n = 1}^\infty \mathscr{P}^{(n)}$.
	\end{Thm}
	
	\begin{proof}
	Our proof uses the sequence $\left( \mathscr{P}^{(n)} \right)_{n = 1}^\infty$ to induce a geometry on $(X, \mu\vert_\mathfrak{P})$, and then uses probabilistic techniques to show that this geometry satisfies the conditions of \Cref{thm:Ergodic Lebesgue Differentiation}. Set $\mathscr{P}^{(0)} = \{X\}$.
	
	Let $\pi : X \to \prod_{n = 1}^\infty \mathscr{P}^{(n)}$ be the measurable map defined by the rule
	\begin{align*}
	x	& \in \pi(x)(n)	& \textrm{for all $x \in X, n \in \N$.}
	\end{align*}
	We can thus induce a (pseudo)metric $\rho : X \times X \to [0, 1]$ by
	\begin{align*}
		\rho(x_1, x_2)	& = \begin{cases}
			\frac{1}{m + 1/2}	& \textrm{if $x_1 \neq x_2, m = \min \left\{ n \in \N : \pi(x_1)(n) \neq \pi(x_2)(n) \right\}$,} \\
			0	& \textrm{otherwise,}
		\end{cases}
	\end{align*}
	meaning that
	\[B(x, 1/n) = \left\{ x' \in X : \forall i \in \{1, \ldots, n - 1\} \; \left( \pi(x)(i) = \pi\left(x'\right)(i) \right) \right\} .\]
	The Borel $\sigma$-algebra induced by $\rho$ will be exactly $\mathfrak{P}$. Let $\mathbf{H}^*$ denote the associated Hardy-Littlewood maximal operator with respect to $\rho$. Then $\mathbf{H}^* f = \sup_{n \geq 0} \mathbb{E} \left[ |f| \mid \mathscr{P}^{(n)} \right]$, and the Doob maximal inequality for martingales (cf. \cite[Theorem 5.2.1]{StroockProbability}) tells us that
	\begin{align*}
		\mu \left( \left\{ x \in X : \mathbf{H}^* f(x) \geq \epsilon \right\} \right)	& \leq \frac{1}{\epsilon} \|f\|_1	& \textrm{for all $f \in L^1(\mu), \epsilon > 0$,}
	\end{align*}
	meaning in particular that $\mu$ has the Hardy-Littlewood property with respect to $\rho$. Therefore, given $f \in L^p(\mu), p > 1$, we can apply \Cref{thm:Ergodic Lebesgue Differentiation} to conclude that
	\begin{align*}
		\lim_{n, k \to \infty} \mathbb{E} \left[ \A_k f \mid \mathscr{P}^{(n)} \right] (x)	& = \lim_{r \searrow 0, k \to \infty} U_{k, r} f (x) = \mathbb{E} \left[ f \mid \mathfrak{I} \cap \mathfrak{P} \right](x)	& \textrm{a.s.}
	\end{align*}
	\end{proof}
	
	\begin{proof}[Proof of \Cref{thm:martingale theorem}]
	This follows immediately from \Cref{thm:stronger martingale theorem} with $\mathscr{P}^{(n)} = \bigvee_{i = 0}^{n - 1} T^{-i} \mathscr{P}$.
	\end{proof}
		
	\section{Open questions}\label{sec:Open questions}
	
	In an earlier version of this article, we posed the following questions left open by \Cref{thm:Ergodic Lebesgue Differentiation}. Since that version of this article was made public, both questions have been resolved by J. Li in \cite{SharpOrliczEndpoints}. In this version, we keep these questions in their original form, but go on to briefly summarize Li's solutions to them.
	
	\begin{Qn}\label{qn:Lebesgue differentiation property}
	Does \Cref{thm:Ergodic Lebesgue Differentiation} still hold if we replace the Hardy-Littlewood property with the following condition: for every $f \in L^1(\mu)$, we have that $\lim_{r \searrow 0} \dashint_{B(x, r)} f \mathrm{d} \mu = f(x)$ a.s.?
	\end{Qn}
	
	This condition -i.e. that the Lebesgue differentiation theorem holds on the space- is necessary, since this is the special case of \Cref{thm:Ergodic Lebesgue Differentiation} where $T$ is the identity map on $X$. However, as we noted in the introduction, this ``Lebesgue differentiation property" is strictly weaker than the Hardy-Littlewood property used to prove \Cref{thm:Ergodic Lebesgue Differentiation}.

	\begin{Qn}\label{qn:L^1}
	Does \Cref{thm:Ergodic Lebesgue Differentiation} hold for all $f \in L^1(\mu)$?
	\end{Qn}
	
	Since the first version of this article was made public, both of these questions have been resolved by J. Li in \cite[Theorem A]{SharpOrliczEndpoints}, who answered \Cref{qn:Lebesgue differentiation property} in the positive and \Cref{qn:L^1} in the negative. Li showed that if $(X, \mu)$ satisfies the Lebesgue differentiation property, then \Cref{thm:Ergodic Lebesgue Differentiation} holds for $f \in L \log L(\mu)$. Moreover, this $L \log L$ bound could not be improved upon, where the sharpness of this bound is made precise through the language of Orlicz spaces. In particular, there exists a Borel map $T : [0, 1] \to [0, 1]$ which preserves the Lebesgue measure $\mu$, and a nonnegative function $f \in L^1(\mu)$ which admits a measurable set $E \subseteq [0, 1]$ for which $\mu(E) > 0$ and
	\begin{align*}
	\limsup_{r \searrow 0, k \to \infty} U_{r, k} f(x)	& = + \infty	& \textrm{for all $x \in E$,}
	\end{align*}
	answering our \Cref{qn:L^1} in the negative (cf. \cite[Theorem 3.6]{SharpOrliczEndpoints}). We also note that \cite[Theorem A]{SharpOrliczEndpoints} implies that our hypothesis in \Cref{thm:martingale theorem,thm:stronger martingale theorem} that $f \in L^p(\mu)$ for some $p > 1$ can be replaced with the hypothesis that $f \in L \log L(\mu)$.
	
	We conclude our discussion with a weaker result along these lines, showing that $U_{r_k, k} f \stackrel{k \to \infty}{\to} f^*$ for any $f \in L^1(\mu)$, provided that $r_k \searrow 0$ satisfies a ``rapid decay" condition that depends on $f$. We note that in light of \cite[Theorem A]{SharpOrliczEndpoints}, this rapid decay condition is indispensable.

	\begin{Prop}\label{prop:Decay condition for one function}
	Consider $f \in L^1(\mu)$. Then there exists a sequence $\left(\delta_k \right)_{k = 1}^\infty$ of positive numbers and a measurable set $X' \subseteq X$ with the following properties:
		\begin{enumerate}
			\item $\mu\left(X'\right) = 1$, and
			\item if $(r_k)_{k = 1}^\infty$ is a sequence with $0 < r_k \leq \delta_k$ for all sufficiently large $k \in \N$, and $x \in X'$, then $\dashint_{B(x, r_k)} \A_k f \mathrm{d} \mu \stackrel{k \to \infty}{\to} f^*(x)$.
		\end{enumerate}
	\end{Prop}
	
	\begin{proof}
		Take $\Omega = \left\{ x \in X : \A_k f(x) \stackrel{k \to \infty}{\to} f^*(x) \right\}$, where $\mu(\Omega) = 1$ by the pointwise ergodic theorem. For $j \in \N \cup \{0\}, n \in \N, h \in \N$, set
		\begin{align*}
			E_{j, n, h}	& = \left\{ x \in X : \sup_{0 < r < 1/h} \left| \dashint_{B(x, r)} f \circ T^j \mathrm{d} \mu - f \left( T^j x \right) \right| \leq 1 / n \right\}.
		\end{align*}
		Then for every fixed $j, n$, we can see that the sequence $\left( E_{j, n, h} \right)_{h = 1}^\infty$ is increasing, and \Cref{lem:Facts about HLP} tells us that $\mu \left( \bigcup_{h = 1}^\infty E_{j, n, h} \right) = 1$. For $k \in \N, 0 \leq j \leq k - 1$, choose $h_{j, k} \in \N$ such that
		\begin{align*}
			\mu \left( E_{j, k, h_{j, k}} \right)	& \geq 1 - \frac{1}{k \cdot 2^k} ,
		\end{align*}
		and set
		\begin{align*}
			F_k	& = \bigcap_{j = 0}^{k - 1} E_{j, k, h_{j, k}} ,	& \delta_k	& = \left( \max_{0 \leq j \leq k - 1} h_{j, k} \right)^{-1} .
		\end{align*}
		Then $\mu \left( X \setminus F_k \right) \leq 2^{-k}$, so the Borel-Cantelli lemma tells us that $\mu \left( \bigcup_{K = 1}^\infty \bigcap_{k = K}^\infty F_k \right) = 1$, i.e. almost every $x \in X$ is in all but finitely many of the $F_k$. Set $X' = \Omega \cap \left( \bigcup_{K = 1}^\infty \bigcap_{k = K}^\infty F_k \right)$.
		
		Now take $(r_k)_{k = 1}^\infty$ to be a sequence such that $0 < r_k \leq \delta_k$ for all $k \geq k_0$, for some $k_0 \in \N$, and consider a point $x \in X'$. Then there exists $k_1 \in \N$ such that $x \in F_k$ for all $k \geq k_1$. Set $K = \max \left\{ k_0, k_1 \right\}$. Then for all $k \geq K$, we have that
		\begin{align*}
			\left| \dashint_{B(x, r_k)} \A_k f \mathrm{d} \mu - f^*(x) \right|	& \leq \left| \dashint_{B(x, r_k)} \A_k f \mathrm{d} \mu - \A_k f(x) \right| + \left| \A_k f(x) - f^*(x) \right| \\
			& \leq \left[ \frac{1}{k} \sum_{j = 0}^{k - 1} \left| \dashint_{B(x, r_k)} f \circ T^j \mathrm{d} \mu - f \left( T^j x \right) \right| \right] + \left| \A_k f(x) - f^*(x) \right| \\
			& \leq \frac{1}{k} + \left| \A_k f(x) - f^*(x) \right|			& \stackrel{k \to \infty}{\to} 0 .
		\end{align*}
	\end{proof}
	
	\appendix
	
	\crefalias{section}{appendix}
	
	\section{Extending \Cref{thm:Ergodic Lebesgue Differentiation} to averages along the squares}\label{sec:exotic averages}
	
	Our goal in this appendix is to extend the techniques used in \Cref{sec:Proof of main theorem} to include other, more ``exotic" types of ergodic averages by proving a version of \Cref{thm:Ergodic Lebesgue Differentiation} where the ergodic averages $\frac{1}{k} \sum_{j = 0}^{k - 1} T^j$ are replaced by the ergodic averages  $\frac{1}{k} \sum_{j = 0}^{k - 1} T^{j^2}$along the squares. While our techniques used here are slightly different from those used to prove \Cref{thm:Ergodic Lebesgue Differentiation}, we will explain afterwards how these techniques can be generalized to other subsequential or weighted ergodic averages.
	
	\begin{Thm}\label{thm:Lebesgue differentiation along Bourgain averages}
	Consider a probability measure-preserving dynamical system $(X, \mu, T)$, where $X$ is a metric space, and $\mu$ is a Borel probability measure on $X$ with the Hardy-Littlewood property. Then for all $f \in L^p(\mu), p > 1$, the limit
	\begin{align*}
		\lim_{r \searrow 0, k \to \infty} \mu(B(x, r))^{-1} \int_{B(x, r)} \frac{1}{k} \sum_{j = 0}^{k - 1} f \circ T^{j^2} \mathrm{d} \mu
	\end{align*}
	exists a.s.
	\end{Thm}
	
	\begin{proof}
	Assume throughout this proof that $1 < p \leq 2$. For this proof, we will write
		\begin{align*}
			\mathbb{B}_{k}	& : = \frac{1}{k} \sum_{j = 0}^{k - 1} T^{j^2} ,	& \mathbf{B}^* f	& : = \sup_{k \in \N} \mathbb{B}_k |f| ,	& V_{r, k} f (x)	& = \dashint_{B(x, r)} \mathbb{B}_k f \mathrm{d} \mu,	& \mathbf{V}^* f(x)	& : = \sup_{r > 0, k \in \N} V_{r, k} |f|(x) .
		\end{align*}
		This proof relies on a spectral analysis of the sequence $(\mathbb{B}_k)_{k = 1}^\infty$, where the $\mathbb{B}_k$ are interpreted as operators $L^2(\mu) \to L^2(\mu)$. Specifically, we have that
		\begin{equation}\label{eq:spectral analysis of Bourgain averages}
		\mathbb{B}_k f \stackrel{k \to \infty}{\to} \sum_{\lambda \in \sigma_{pp}(T)} c(\lambda) P_\lambda f ~ \textrm{a.s. for all $f \in L^2(\mu)$},
		\end{equation}
		where $\sigma_{pp}(T) \subseteq \{\lambda \in \C: |\lambda| = 1\}$ is the pure point spectrum of $T$, i.e. the set of eigenvalues of $T$, where $c(\lambda) = \lim_{k \to \infty} \frac{1}{k} \sum_{j = 0}^{k - 1} \lambda^{j^2}$, and where $P_\lambda : L^2(\mu) \to L^2(\mu)$ denotes the orthogonal projection of $L^2(\mu)$ onto the eigenspace $\mathcal{E}_\lambda = \ker(T - \lambda) \subseteq L^2(\mu)$ (cf. \cite{BourgainPolynomials}).
		
		The two main ingredients of this proof are the same as in our proof of \Cref{thm:Ergodic Lebesgue Differentiation}, namely convergence for a dense subspace of $L^p(\mu)$ and a maximal inequality for $\mathbf{V}^*$ to extend this convergence to all of $L^p(\mu)$. The maximal inequality for $\mathbf{V}^*$ can be derived in much the same way as we derived the maximal inequality for $\mathbf{U}^*$ in \Cref{lem:Facts about HLP}, since it is known that there exist constants $B_p$ for $p > 1$ such that $\left\| \mathbf{B}^* f \right\|_p \leq B_p \|f\|_p$ for all $f \in L^p(\mu)$ (cf. \cite[Theorem 1]{BourgainPolynomials}), and $\mathbf{V}^* \leq \mathbf{H}^* \mathbf{B}^*$. That is to say that for all $p > 1$, there exists a constant $M_p \geq 1$ such that
		\begin{equation}\label{eq:maximal inequality for Bourgain Lebesgue averages}
			\left\| \mathbf{V}^* f \right\| \leq M_p \|f\|_p ~ \textrm{for all $f \in L^p(\mu)$.} 
		\end{equation}
		This leaves us the task of showing that $V_{r, k} f$ converges a.s. for all $f \in L^\infty (\mu)$.
		
		Let $\mathfrak{K}$ denote the Kronecker factor of $(X, \mu, T)$, i.e. the $\sigma$-algebra on $X$ generated by the eigenfunctions of $T$. For $f \in L^\infty(\mu)$, write
		\begin{align*}
			g	& = \mathbb{E} \left[ f \mid \mathfrak{K} \right] = \sum_{\lambda \in \sigma_{pp}(T)} P_\lambda f ,	& h	& = f - g .
		\end{align*}
		Note in particular that if $f \in L^\infty(\mu)$, then $g = \mathbb{E} \left[ f \mid \mathfrak{K} \right] \in L^\infty(\mu)$, and therefore $h = f - g \in L^\infty(\mu)$.
		
		\textbf{Step 1 - analyzing $g$:} We will prove the following claim: that if $\phi \in \bigoplus_{\lambda \in \sigma_{pp}(T)} \mathcal{E}_\lambda$, then $V_{r, k} \phi \stackrel{r \searrow 0, k \to \infty}{\to} \sum_{\lambda \in \sigma_{pp}(T)} c(\lambda) P_\lambda \phi$ a.s.
		
		Suppose that $\psi \in \mathcal{E}_\lambda \cap L^\infty(\mu)$ for some $\lambda \in \sigma_{pp}(T)$. Then
		\begin{align*}
			\left| V_{r, k} \psi(x) - c(\lambda) \psi(x) \right|	& \leq \left| \dashint_{B(x, r)} \left( \left( \frac{1}{k} \sum_{j = 0}^{k - 1} \lambda^{j^2} \right) - c(\lambda) \right) \psi \mathrm{d} \mu \right| + |c(\lambda)| \left| \dashint_{B(x, r)} \psi \mathrm{d} \mu - \psi(x) \right| \\
			& \leq \left| \left( \frac{1}{k} \sum_{j = 0}^{k - 1} \lambda^{j^2} \right) - c(\lambda) \right| \cdot \|\psi\|_\infty +  \left| \dashint_{B(x, r)} \psi \mathrm{d} \mu - \psi(x) \right| \\
			& \stackrel{r \searrow 0, k \to \infty}{\to} 0	& \textrm{a.s.}
		\end{align*}
		Because $\mathcal{E}_\lambda \cap L^\infty(\mu)$ is dense in $\mathcal{E}_\lambda$ for all $\lambda \in \sigma_{pp}(T)$ (cf. \cite[Lemma 17.3]{EisnerOperators}), we can use \eqref{eq:maximal inequality for Bourgain Lebesgue averages} again to say that if $\phi \in \bigoplus_{\lambda \in \sigma_{pp}(T)} \mathcal{E}_\lambda$, then
		\begin{align*}
			V_{r, k} \phi(x)	& \stackrel{r \searrow 0, k \to \infty}{\to} \sum_{\lambda \in \sigma_{pp}(T)} c(\lambda) P_{\lambda} \phi	& \textrm{a.s.}
		\end{align*}
		In particular, this means that $V_{r, k} g \stackrel{r \searrow 0, k \to \infty}{\to} \sum_{\lambda \in \sigma_{pp}(T)} c(\lambda) P_\lambda g$ a.s.
		
		\textbf{Step 2 - analyzing $h$:} We want to show that $V_{r, k} h \stackrel{r \searrow 0, k \to \infty}{\to} 0$ a.s. Assume without loss of generality that $\|h\|_\infty \leq 1$, and fix $\epsilon > 0$. Observe that
		\begin{align*}
			\left\{ x \in X : \limsup_{r \searrow 0, k \to \infty} \left| V_{r, k} h(x) \right| \geq \epsilon \right\}	& \subseteq \bigcap_{K = 1}^\infty \left\{ x \in X : \sup_{r > 0} \sup_{k \geq K} \left| V_{r, k} h(x) \right| \geq \epsilon \right\} .
		\end{align*}
		Set $E_{K} = \left\{ x \in X : \sup_{k \geq K} \left| \mathbb{B}_k h (x) \right| \geq \epsilon / 2 \right\}$. We know from \eqref{eq:spectral analysis of Bourgain averages} that $\mu(E_{K}) \stackrel{K \to \infty}{\to} 0$. Therefore
		\begin{align*}
			\sup_{r > 0} \sup_{k \geq K} \left| V_{r, k} h(x) \right|	& = \sup_{r > 0} \dashint_{B(x, r)} \sup_{k \geq K} \left| \B_k h\right| \mathrm{d} \mu  \\
			& \leq \left[ \sup_{r > 0} \dashint_{B(x, r)} \chi_{E_K} \sup_{k \geq K} |\B_k h| \mathrm{d} \mu \right] + \left[ \dashint_{B(x, r)} \chi_{X \setminus {E_K}} \sup_{k \geq K} |\B_k h| \mathrm{d} \mu \right] \\
			& \leq \sup_{r > 0} \dashint_{B(x, r)} \chi_{E_K} \|h\|_\infty \mathrm{d} \mu + \sup_{r > 0} \dashint_{B(x, r)} \frac{\epsilon}{2} \mathrm{d} \mu	\\		& \leq \left[ \mathbf{H}^* \chi_{E_K} (x) \right] + \frac{\epsilon}{2}	.
		\end{align*}
		Thus for $K \in \N$, we see
		\begin{align*}
			\mu \left( \left\{ x \in X : \limsup_{r \searrow 0, k \to \infty} \left| V_{r, k} h(x) \right| \geq \epsilon \right\} \right)	& \leq \mu \left( \left\{ x \in X : \mathbf{H}^* \chi_{E_K} (x) \geq \epsilon / 2 \right\} \right) \\
				& \leq \left( \frac{2}{\epsilon} \right)^2 \left\| \mathbf{H}^* \chi_{E_K} \right\|_2^2	& \leq \frac{4 C_2^2}{\epsilon^2} \mu(E_{K})^2 ,
		\end{align*}
		where $C_2$ is as in \Cref{lem:Facts about HLP}. By taking $K \to \infty$, we can conclude that $\limsup_{r \searrow 0, k \to \infty} \left| V_{r, k} h(x) \right| < \epsilon$ a.s., and since our choice of $\epsilon > 0$ was arbitrary, it follows that $V_{r, k} h \stackrel{r \searrow 0, k \to \infty}{\to} 0$ a.s.
		
		\textbf{Step 3 - the density argument:} We've already established that if $f \in L^\infty(\mu)$, then \linebreak$\lim_{r \searrow 0, k \to \infty} V_{r, k} f $ exists a.s. Since $L^\infty (\mu)$ is dense in $L^p(\mu)$ for all $p > 1$, we can use \eqref{eq:maximal inequality for Bourgain Lebesgue averages} to extend the convergence to all $f \in L^p(\mu)$. This is essentially the same argument as we used in our proof of \Cref{thm:Ergodic Lebesgue Differentiation}, so we won't repeat it.
	\end{proof}
	
	The characterization \eqref{eq:spectral analysis of Bourgain averages} of the limit $\lim_{k \to \infty} \B_k$ is the heart of our argument. Similar characterizations are available for other weighted ergodic averages and subsequential ergodic averages (cf. \cite[Theorems 21.2 and 21.14]{EisnerOperators}), meaning the techniques used in our proof of \Cref{thm:Lebesgue differentiation along Bourgain averages} can be generalized to so-called ``pointwise good" sequences and weights, as defined in \cite[Chapter 21]{EisnerOperators}. As such, \Cref{thm:Lebesgue differentiation along Bourgain averages} can be seen as a stand-in for a more general theorem about a wider class of ergodic averages.
	
	\section*{Acknowledgments}
	
	We thank Itai Bar Deroma for some helpful comments on and questions about an earlier version of this project, as well as Tom Meyerovitch for some helpful discussions. This research was supported by Israel Science Foundation grant no. 985/23.
	
	\bibliography{/Users/AJYou/OneDrive/Desktop/Research_stuff/Bibliography/Bibliography}

\begin{bibdiv}
\begin{biblist}

\bib{Assani-Young}{article}{
      author={Assani, Idris},
      author={Young, Aidan},
       title={Spatial-temporal differentiation theorems},
        date={2022},
     journal={Acta Mathematica Hungarica},
      volume={168},
       pages={301\ndash 344},
}

\bib{Bogachev1}{book}{
      author={Bogachev, Vladimir~Igorevich},
       title={Measure theory},
   publisher={Springer},
        date={2007},
      volume={1},
}

\bib{BourgainPolynomials}{article}{
      author={Bourgain, Jean},
       title={Pointwise ergodic theorems for arithmetic sets},
        date={1989},
     journal={Publications Math{\'e}matiques de l'IH{\'E}S},
      volume={69},
       pages={5\ndash 41},
}

\bib{Guzman}{book}{
      author={De~Guzman, Miguel},
       title={Differentiation of integrals in $\mathbb{R}^n$},
organization={Springer Verlag},
        date={1976},
}

\bib{EisnerOperators}{book}{
      author={Eisner, Tanja},
      author={Farkas, B{\'a}lint},
      author={Haase, Markus},
      author={Nagel, Rainer},
       title={Operator theoretic aspects of ergodic theory},
   publisher={Springer},
        date={2015},
      volume={272},
}

\bib{HeinonenLectures}{book}{
      author={Heinonen, Juha},
       title={Lectures on analysis on metric spaces},
   publisher={Springer Science \& Business Media},
        date={2001},
}

\bib{Krengel}{book}{
      author={Krengel, Ulrich},
       title={Ergodic theorems},
   publisher={Walter de Gruyter},
        date={1985},
      volume={6},
}

\bib{SharpOrliczEndpoints}{misc}{
      author={Li, Jie},
       title={Sharp orlicz endpoints for spatial-temporal ergodic averaging},
        date={2026},
         url={https://arxiv.org/abs/2608.03767},
}

\bib{SawyerMaximalInequality}{article}{
      author={Sawyer, Stanley},
       title={Maximal inequalities of weak type},
        date={1966},
     journal={Annals of Mathematics},
      volume={84},
      number={2},
       pages={157\ndash 174},
}

\bib{SjogrenMaximalFunction}{article}{
      author={Sj{\"o}gren, Peter},
       title={A remark on the maximal function for measures in rn},
        date={1983},
     journal={American Journal of Mathematics},
       pages={1231\ndash 1233},
}

\bib{StroockProbability}{book}{
      author={Stroock, Daniel~W},
       title={Probability theory: an analytic view},
   publisher={Cambridge university press},
        date={2010},
}

\end{biblist}
\end{bibdiv}
	
\end{document}